\documentclass[psamsfonts]{amsart}   
\usepackage{amssymb}

\usepackage[dvips]{graphicx}

\usepackage{graphicx}
\usepackage{amssymb}                       
\usepackage{amsmath}
\usepackage{color}
\usepackage{times}

\usepackage[unicode,bookmarks,colorlinks]{hyperref}
\hypersetup{
    linkcolor=brickred,
}

\definecolor{mahogany}{cmyk}{0, 0.77, 0.87, 0}
\definecolor{salmon}{cmyk}{0, 0.53, 0.38, 0}
\definecolor{melon}{cmyk}{0, 0.46, 0.50, 0}
\definecolor{yellowgreen}{cmyk}{0.44, 0, 0.74, 0}
\definecolor{brickred}{cmyk}{0, 0.89, 0.94, 0.28}
\definecolor{OliveGreen}{cmyk}{0.64, 0, 0.95, 0.40}
\definecolor{RawSienna}{cmyk}{0, 0.72, 1.0, 0.45}
\definecolor{ZurichRed}{rgb}{1, 0, 0} 

\usepackage{fancyhdr}
\usepackage{amsmath,amstext,amssymb,amsopn,amsthm}
\usepackage{amsmath,amssymb,amsthm}
\usepackage[mathscr]{eucal}

\begin{document}

\newtheorem{lemma}[thm]{Lemma}

\newtheorem{proposition}{Proposition}
\newtheorem{theorem}{Theorem}[section]
\newtheorem{deff}[thm]{Definition}
\newtheorem{case}[thm]{Case}
\newtheorem{prop}[thm]{Proposition}
\newtheorem{example}{Example}

\newtheorem{corollary}{Corollary}

\theoremstyle{definition}
\newtheorem{remark}{Remark}

\numberwithin{equation}{section}
\numberwithin{definition}{section}
\numberwithin{corollary}{section}

\numberwithin{theorem}{section}

\numberwithin{remark}{section}
\numberwithin{example}{section}
\numberwithin{proposition}{section}

\newcommand{\gap}{\lambda_{2,D}^V-\lambda_{1,D}^V}
\newcommand{\gapR}{\lambda_{2,R}-\lambda_{1,R}}
\newcommand{\bD}{\mathrm{I\! D\!}}
\newcommand{\calA}{\mathcal{A}}
\newcommand{\calD}{\D}

\newcommand{\conjugate}[1]{\overline{#1}}
\newcommand{\abs}[1]{\left| #1 \right|}
\newcommand{\cl}[1]{\overline{#1}}
\newcommand{\expr}[1]{\left( #1 \right)}
\newcommand{\set}[1]{\left\{ #1 \right\}}

\newcommand{\calC}{\mathcal{C}}
\newcommand{\calE}{\mathcal{E}}
\newcommand{\calF}{\mathcal{F}}
\newcommand{\Rd}{\mathbb{R}^d}
\newcommand{\BR}{\mathcal{B}(\Rd)}
\newcommand{\R}{\mathbb{R}}
\newcommand{\al}{\alpha}
\newcommand{\RR}[1]{\mathbb{#1}}
\newcommand{\bR}{\mathrm{I\! R\!}}
\newcommand{\ga}{\gamma}
\newcommand{\om}{\omega}
\newcommand{\A}{\mathbb{A}}
\newcommand{\bH}{\mathbb{H}}

\newcommand{\bb}[1]{\mathbb{#1}}
\newcommand{\bI}{\bb{I}}
\newcommand{\bN}{\bb{N}}

\newcommand{\uS}{\mathbb{S}}
\newcommand{\M}{{\mathcal{M}}}
\newcommand{\calB}{{\mathcal{B}}}

\newcommand{\W}{{\mathcal{W}}}

\newcommand{\m}{{\mathcal{m}}}

\newcommand {\mac}[1] { \mathbb{#1} }

\newcommand{\bC}{\Bbb C}
\newcommand{\D}{\mathbb{D}}

\newtheorem{rem}[theorem]{Remark}
\newtheorem{dfn}[theorem]{Definition}
\theoremstyle{definition}
\newtheorem{ex}[theorem]{Example}
\numberwithin{equation}{section}

\newcommand{\Pro}{\mathbb{P}}
\newcommand\F{\mathcal{F}}
\newcommand\E{\mathbb{E}}
\newcommand\e{\varepsilon}
\def\H{\mathcal{H}}
\def\t{\tau}

\newcommand{\blankbox}[2]{%
  \parbox{\columnwidth}{\centering
    \setlength{\fboxsep}{0pt}%
    \fbox{\raisebox{0pt}[#2]{\hspace{#1}}}%
  }%
}
\subjclass[2010]{60G40, 60G44, 31B05}

\keywords{Martingale, Supermartingale, maximal, harmonic function, best constant}

\title[Martingale inequalities]
 {Sharp maximal $L^p$-estimates for martingales}

\author{Rodrigo Ba\~nuelos}\thanks{R. Ba\~nuelos is supported in part  by NSF Grant
\# 0603701-DMS}
\address{Department of Mathematics, Purdue University, West Lafayette, IN 47907, USA}
\email{banuelos@math.purdue.edu}
\author{Adam Os\c ekowski}
\address{Department of Mathematics, Informatics and Mechanics, University of Warsaw, Banacha 2, 02-097 Warsaw, Poland}
\email{ados@mimuw.edu.pl}
\thanks{A. Os\c ekowski is supported in part by NCN grant DEC-2012/05/B/ST1/00412.}

\begin{abstract}
Let $X$ be a supermartingale starting from $0$ which has only nonnegative jumps. For each $0<p<1$  we determine the best constants $c_p$, $C_p$ and $\mathfrak{c}_p$ such that
$$ \,\,\,\,\sup_{t\geq 0}\left|\left|X_t\right|\right|_p\leq C_p\left|\left|-\inf_{t\geq 0}X_t\right|\right|_p,$$
$$ \,\,||\sup_{t\geq 0}X_t||_p\leq c_p\left|\left|-\inf_{t\geq 0}X_t\right|\right|_p$$
and
$$ ||\sup_{t\geq 0}|X_t|\;||_p\leq \mathfrak{c}_p\left|\left|-\inf_{t\geq 0}X_t\right|\right|_p.$$
The estimates are shown to be sharp if $X$ is assumed to be a stopped one-dimensional Brownian motion. 
The inequalities are deduced from the existence of special functions, enjoying certain majorization and convexity-type properties. Some applications concerning harmonic functions on Euclidean domains are indicated.
\end{abstract}

\maketitle

\section{Introduction}  

Suppose that $(\Omega,\F,\mathbb{P})$ is a complete probability space, filtered by $(\F_t)_{t\geq 0}$, a nondecreasing family of sub-$\sigma$-algebras of $\F$, such that $\F_0$ contains all the events of probability $0$. Assume further that $X=(X_t)_{t\geq 0}$ is a martingale on this filtration  with right-continuous trajectories that have limits from the left. Define the associated supremum, infimum and both-sided supremum processes $M^+=(M^+_t)_{t\geq 0}$, $M^-=(M^-_{t})_{t\geq 0}$ and $M=(M_t)_{t\geq 0}$ by the formulas
$$ M^+_t=\sup_{0\leq s\leq t} X_s\vee 0,\qquad M^-_{t}=\inf_{0\leq s\leq t} X_s\wedge 0$$
and
$$  M_t=\sup_{0\leq s\leq t}|X_t|=M_t^-\vee M_t^+,$$
where, as usual, $a\vee b=\max\{a,b\}$ and $a\wedge b=\min\{a,b\}$.  
We will also use the notation $\Delta X_t$ for $X_t-X_{t-}$, the jump of $X$ at time $t$ (we assume that $X_{0-}=0$ almost surely). 

The inequalities involving various sizes of $X$, $M^+$, $M^-$ and $M$ have played an important role in probability, especially in the theory of stochastic processes and stochastic integration. For instance, recall the classical result of Doob \cite{Do}: we have
$$ ||M||_p\leq \frac{p}{p-1}||X||_p,\qquad 1<p<\infty, $$ 
and the constant $p/(p-1)$ is the best possible, even in the weaker estimates
$$ ||M^+||_p\leq \frac{p}{p-1}||X||_p,\quad ||M^-||_p\leq \frac{p}{p-1}||X||_p.$$
Here and below, we will use the convention $||Y||_p=\sup_{t\geq 0}\left(\E |Y_t|^p\right)^{1/p}$ for any semimartingale $Y=(Y_t)_{t\geq 0}$. 
For $p=1$, the above $L^p$ bound does not hold with any finite constant, but we have the following sharp LlogL estimate (see Gilat \cite{Gi} and Peskir \cite{Pe}): for any $K>1$,
$$ ||M||_1 \leq K \sup_{t\geq 0}\E |X_t|\log^+|X_t|+L(K),$$
where $ L(K)=1+(e^K(K-1))^{-1}$ is the best possible. There are many versions of these results, and we refer the interested reader to the monograph \cite{PS} by Peskir and Shiryaev for an overview, presenting the subject from the viewpoint of optimal stopping theory.

We will be particularly interested in sharp $L^p$ bounds involving $X$, $M^+$, $M^-$ and $M$ in the case $0<p<1$. It is well-known that in general such estimates do not hold with any finite constants unless we assume some additional regularity of the trajectories or the distribution such as continuity, conditional symmetry or nonnegativity. For instance, we have the following result, proved by Shao in \cite{Shao}.

\begin{theorem}
Suppose that $X$ is a nonnegative martingale. Then for any $0<p<1$ we have the sharp bound
$$ ||M||_p\leq \frac{1}{(1-p)^{1/p}}||X||_p.$$
\end{theorem} 

We will be interested in a slightly different class of estimates. 
Motivated by the study Hardy's  $H^p$ spaces for harmonic functions on the upper half-space $\R^{n+1}_{+}$,  Burkholder  obtained the following result.  

\begin{theorem}[Burkholder \cite{Bur1}]\label{Burkholder}
Suppose that $X$ is a martingale with continuous sample paths with $X_0=0$.  
  If $\Phi$ is a nondecreasing continuous function on $[0,\infty)$ such that $\Phi(0)=0$ and $\Phi(\beta\lambda)\leq \gamma\Phi(\lambda)$, for some $\beta>\gamma>1$ and all $\lambda>0$, then 
\begin{equation}\label{burk1}
\sup_{t\geq 0}\E\Phi(M_t)\leq C\sup_{t\geq 0}\E\Phi(-M_t^{-}),
\end{equation}
where $C$ depends only on $\beta$ and $\gamma$.  In particular, if $0<p<1$, then
\begin{equation}\label{burk2}
\|M\|_p\leq C_p\|M^{-}\|_p.
\end{equation} 
\end{theorem}
As Burkholder points out (see his Example 6.3 in \cite{Bur1}), by stopping Brownian motion at the first time it hits $-1$, it follows that \eqref{burk2} does not hold for $p\geq 1$.

Burkholder's proof of \eqref{burk1} uses good-$\lambda$ inequalities.  Over the years other proofs of \eqref{burk2} have been given, including the recent one in \cite{Sch} which is written in terms of the functions $M^+$ and $M^-$.   The inequality in \cite{Sch} is applied to prove a stochastic Gronwall Lemma. 
   The goal of this paper is to obtain the best constant in \eqref{burk2} and its variant proved in \cite{Sch}.  Actually, we will go much further and study a wider class of processes: our reasoning will enable us to obtain sharp estimates for supermartingales which do not have negative jumps. 
In the formulation of our main results, we will need some additional constants. A straightforward analysis of a derivative shows that there is unique $p_0\in (0,1)$ for which 
\begin{equation}\label{defp0}
 p_0^{-1}-1=\left(p_0^{-1}-1\right)^{1-p_0}+1.
 \end{equation}
Computer simulations show that $p_0\simeq 0.1945\ldots$. Now, if $p\in (0,p_0]$, let
\begin{equation}\label{defalfa0}
 \alpha_p=\left(\frac{1-p}{p}\right)^{1-p}\qquad \mbox{and}\qquad C_p=\frac{1-p}{p}.
 \end{equation}
On the other hand, if $p\in (p_0,1)$, let $\alpha_p$ be the unique solution to the equation
\begin{equation}\label{defalfa}
(1-p)\left(\alpha_p^{1/(1-p)}+1\right)=\alpha_p+1
\end{equation}
and set $C_p=(1+\alpha_p^{-1})^{1/p}$. Next, for any $p\in (0,1)$, let
$$ c_p=\left(\left(\frac{1}{p}-1\right)^p+\int_{p^{-1}-1}^\infty \frac{s^{p-1}}{s+1}\mbox{d}s\right)^{1/p}.$$
Finally, introduce the constant $\mathfrak{c}_p$ by
$$ \mathfrak{c}_p=\begin{cases}
\displaystyle \left(\left(\frac{1}{p}-1\right)^p+\int_{p^{-1}-1}^\infty \frac{s^{p-1}}{s+1}\mbox{d}s\right)^{1/p} & \mbox{if }0<p\leq 1/2,\\
\displaystyle \left(1+\int_1^\infty \frac{s^{p-1}}{s+1}\mbox{d}s\right)^{1/p} & \mbox{if }1/2<p<1.
\end{cases}$$
Observe that $\mathfrak{c}_p=c_p$ for $0<p\leq 1/2$; on the other hand, when $p\in (1/2,1)$, the constant $\mathfrak{c}_p$ is easily seen to be larger (which will also be clear from the reasoning below). 

We are ready to state the results.  The first theorem gives a sharp comparison of $L^p$ norms of a supermartingale and its infimum process. 

\begin{theorem}\label{theorem1} 
Suppose that $X$ is an adapted supermartingale with only nonnegative jumps, satisfying $X_0=0$ almost surely. Then for any $0<p<1$ we have
\begin{equation}\label{mainin0}
 ||X||_p\leq C_p||M^-||_p
\end{equation}
and the constant $C_p$ is the best possible. It is already the best possible if $X$ is assumed to be a stopped Brownian motion.
\end{theorem}

The second result compares the sizes of the supremum and the infimum processes.

\begin{theorem}\label{theorem2} 
Suppose that $X$ is an adapted supermartingale with only nonnegative jumps, satisfying $X_0=0$ almost surely. Then for any $0<p<1$ we have
\begin{equation}\label{mainin}
 ||M^+||_p\leq c_p||M^-||_p
\end{equation}
and the constant $c_p$ is the best possible. It is already the best possible if $X$ is assumed to be a stopped Brownian motion.
\end{theorem}

Our final result is a sharp version of the bounds \eqref{mainin0} and \eqref{mainin}, with two-sided maximal function on the left.  This gives the best constant in Burkholder's estimate \eqref{burk2}. Here is the precise statement.

\begin{theorem}\label{theorem3}
Suppose that $X$ is an adapted supermartingale with only nonnegative jumps, satisfying $X_0=0$ almost surely. Then for any $0<p<1$ we have
\begin{equation}\label{mainin2}
 ||M||_p\leq \mathfrak{c}_p||M^-||_p
\end{equation}
and the constant $\mathfrak{c}_p$ is the best possible. It is already the best possible if $X$ is assumed to be a stopped Brownian motion.
\end{theorem}

A few words about the approach and the organization of the paper are in order. Our proofs of \eqref{mainin0}, \eqref{mainin} and \eqref{mainin2} rest on the existence of a certain special functions, and have their roots in the theory of optimal stopping. We present them in the next section. In Section \ref{sharpness} we address the optimality of the constants $C_p$, $c_p$ and $\mathfrak{c}_p$. The final section is devoted to the discussion on related results arising in harmonic analysis on Euclidean domains.


\section{Proofs of \eqref{mainin0}, \eqref{mainin} and \eqref{mainin2}}

Throughout this section $p$ is a fixed number belonging to $(0,1)$. The contents of this section is split naturally into three parts.

\subsection{Proof of \eqref{mainin0}} As announced above, the argument depends heavily on the existence of an appropriate special function. Consider $U:\R\times (-\infty,0]\to \R$ defined by the formula
$$ U(x,z)=\alpha_p^{-1}(-z)^{p-1}\big[px-(p-1)z\big],$$
where $\alpha_p$ is given by \eqref{defalfa0} or \eqref{defalfa}, depending upon the range of  $p$.

\begin{lemma}\label{lemma0}
The function $U$ has the following properties.

(i) It is of class $C^\infty$ on $\R\times (-\infty,0)$.

(ii) For any $z<0$ we have
\begin{equation}\label{uz0}
U_x(x,z)\geq 0\quad \mbox{and}\quad U_z(z,z)= 0.
\end{equation}

(iii) If $x\geq z$, then for any $d\geq 0$ we have
\begin{equation}\label{jump0}
U(x+d,z)= U(x,z)+U_x(x,z).
\end{equation}

(iv) If $x\geq z$, then
\begin{equation}\label{maj0}
U(x,z)\geq |x|^p-C_p^p(-z)^p.
\end{equation}
\end{lemma}
\begin{proof}
The first three parts are evident.  The only nontrivial statement is the majorization \eqref{maj0}. By homogeneity, it is enough to show it for $z=-1$. Let $\varphi(x)=U(x,-1)$ and  $\psi(x)=|x|^p-C_p^p$ for $x\geq -1$. The desired bound follows at once from the following four observations:
\begin{align}
&\varphi\mbox{ is increasing},\\
&\psi\mbox{ is concave, decreasing on }(z,0), \mbox{ and concave, increasing on }(0,\infty),\\
&\varphi(\alpha_p^{1/(1-p)})=\psi(\alpha_p^{1/(1-p)}) \mbox{ and }\varphi'(\alpha_p^{1/(1-p)})=\psi'(\alpha_p^{1/(1-p)}),\\
&\varphi(-1)\geq \psi(-1).\label{init}
\end{align}
The first three conditions are clear and  follow from straightforward computations. The latter observation employs  the definition of $p_0$. Indeed, if $p\leq p_0$, then $p^{-1}-1\geq \left(p^{-1}-1\right)^{1-p}+1$, which is equivalent \eqref{init}. On the other hand, if $p>p_0$, then the definitions of $\alpha_p$ and $C_p$ guarantee that we actually have equality in \eqref{init}. 
This finishes the proof.
\end{proof}

\begin{proof}[Proof of \eqref{mainin0}]
Let $X=(X_t)_{t\geq 0}$ be an adapted supermartingale starting from $0$, which admits only nonnegative jumps, and let $\e>0$ be a fixed parameter. In view of Lemma \ref{lemma0} (i), we may apply It\^o's formula to $U$ and the process $Z^\e=((X_t,M^-_t\wedge (-\e)))_{t\geq 0}$. (We refer the reader to Protter \cite{Pro} for the general It\^o formula used here.) As the result of this application, we get that for each $t\geq 0$,
\begin{equation}\label{ito6}
U(Z_t^\e)=I_0+I_1+I_2+\frac{I_3}{2}+I_4,
\end{equation}
where
\begin{align*}
I_0&=U(Z_0^\e)=U(0,-\e),\\
I_1&=\int_{0+}^t U_x(Z_{s-}^\e)\mbox{d}X_s,\\
I_2&=\int_{0+}^tU_z(Z_{s-}^\e)\mbox{d}\big(M_{s-}^{-}\wedge(-\e)\big),\\
I_3&=\int_{0+}^t U_{xx}(Z_{s-}^\e)\mbox{d}[X,X]_s^c,\\
I_4&=\sum_{0<s\leq t} \Big[U(Z_s^\e)-U(Z^\e_{s-})-U_x(Z^\e_{s-})\Delta X_s\Big].
\end{align*}
Note that due to the assumption on the jumps of $X$, the process $M^-$ is continuous; in particular, this explains why there is no summand  $U_z(Z^\e_{s-})\Delta \big(M_{s}^-\wedge (-\e)\big)$ in $I_4$.

Now, let us analyze the behavior of the terms $I_1$ through $I_4$ separately. The first of them has nonpositive expectation, by the properties of stochastic integrals. Indeed, if $X=N+A$ is  the Doob-Meyer decomposition for $X$ (see e.g. Protter \cite{Pro}), then we have
$$ I_1=\int_{0+}^t U_x(Z_{s-}^\e)\mbox{d}N_s+\int_{0+}^t U_x(Z_{s-}^\e)\mbox{d}A_s.$$
Now the first term has mean zero, while the second integral is nonpositive, because of the first inequality in \eqref{uz} and the fact that $A$ is a nonincreasing process. 
 To deal with $I_2$, we make use of the second condition in \eqref{uz0}. By the aforementioned continuity of $M^-$, we see that the process $M_{s-}^{-}\wedge (-\e)$ decreases only when $X_{s}=M_s^-$, i.e., when the coordinates of the variable $Z_{s-}^\e$ are equal. Then, as we have proved in \eqref{uz0}, we have $U_z(Z_{s-}^\e)=0$ and hence the integral  $I_2$ is zero. The term $I_3$ also vanishes, since for a fixed $z$, the function $x\mapsto U(x,z)$ is linear. Finally, each summand in $I_4$ is  zero: this is guaranteed by \eqref{jump0} and the assumption that $X$ has only nonnegative jumps. Thus, putting all the above facts together and plugging them into \eqref{ito6}, we obtain
$$ \E U(Z_t^\e)\leq U(0,-\e),$$
or, in view of \eqref{maj0},
$$ \E |X_t|^p\leq C_p^p\E \left(-\big(M^-_{t}\wedge (-\e)\big)\right)^p+\alpha_p^{-1}(p-1)\e^p.$$
Letting $\e\to 0$ gives $\E |X_t|^p\leq C_p^p\E (-M^-_{t})^p,$  
and it remains to take the supremum over $t$ to obtain \eqref{mainin0}.
\end{proof}

\subsection{Proof of \eqref{mainin}} Here the reasoning will be more involved. In particular, due to the appearance of the supremum process in \eqref{mainin}, we are forced to consider special functions of \emph{three} variables (corresponding to $X$, $M^+$ and $M^-$). Introduce $U:\R \times [0,\infty)\times (-\infty,0]\to \R$, given by
$$ U(x,y,z)=y^p-c_p^p(-z)^p+p(x-z)(-z)^{p-1}\int_{-y/z}^\infty \frac{r^{p-1}}{r+1}\mbox{d}r$$
if $y>-\left(\frac{1}{p}-1\right)z$, and
$$ U(x,y,z)=\left(\left(\frac{1}{p}-1\right)^p-c_p^p\right)(-z)^p+p(x-z)(-z)^{p-1}\int_{p^{-1}-1}^\infty \frac{r^{p-1}}{r+1}\mbox{d}r$$
if $y\leq -\left(\frac{1}{p}-1\right)z$. Let us prove some important facts concerning this object; they are gathered in the following statement, which can be regarded as the analogue of Lemma \ref{lemma0}.

\begin{lemma}\label{lemma1}
The function $U$ enjoys the following properties.

(i) For all $z<0<y$, the function $U(\cdot,y,z): x\mapsto U(x,y,z)$ is of class $C^2$ and the partial derivatives $U_y(y,y,z)$, $U_z(z,y,z)$ exist.

(ii) For all $z<x<y$ we have
\begin{equation}\label{uz}
 U_x(x,y,z)\geq 0,\qquad U_y(y,y,z)= 0\qquad \mbox{and}\qquad U_z(z,y,z)\geq 0.
\end{equation}

(iii) If $z\leq x\leq y$, then for any $d\geq 0$ we have the bound
\begin{equation}\label{jump}
U(x+d,(x+d)\vee y,z)\leq U(x,y,z)+U_x(x,y,z)d.
\end{equation}

(iv) If $x\geq z$ and $y\geq 0$, then
\begin{equation}\label{maj}
U(x,y,z)\geq y^p-c_p^p(-z)^p.
\end{equation}
\end{lemma}
\begin{proof}
(i) This is straightforward; we leave the verification to the reader.

(ii) The estimate for $U_x$ is evident. The identity $U_y(y,y,z)=0$ is also immediate, both for $y>-\left(\frac{1}{p}-1\right)z$ and $y\leq -\left(\frac{1}{p}-1\right)z$. To show the estimate for $U_z$, note that if $y\leq -\left(\frac{1}{p}-1\right)z$, then
\begin{align*}
 U_z(z,y,z)=-p\left(\left(\frac{1}{p}-1\right)^p-c_p^p\right)(-z)^{p-1}-p(-z)^{p-1}\int_{p^{-1}-1}^\infty \frac{r^{p-1}}{r+1}\mbox{d}r=0,
\end{align*}
by the formula for $c_p$. On the other hand, if $y> -\left(\frac{1}{p}-1\right)z$, we easily derive that
\begin{align*}
 U_z(z,y,z)&=pc_p^p(-z)^{p-1}-p(-z)^{p-1}\int_{-y/z}^\infty \frac{r^{p-1}}{r+1}\mbox{d}r\\
 &> pc_p^p(-z)^{p-1}-p(-z)^{p-1}\int_{p^{-1}-1}^\infty \frac{r^{p-1}}{r+1}\mbox{d}r=0,
\end{align*}
where the latter equality is again due to the definition of $c_p$.

(iii) If $x+d\leq y$, then both sides are equal, because of the linearity of $x\mapsto U(x,y,z)$. Therefore, suppose that $x+d>y$ and consider the following cases. If $x+d\leq -\left(\frac{1}{p}-1\right)z$, then also $y\leq -\left(\frac{1}{p}-1\right)z$ and again \eqref{jump} becomes an equality. If $y< -\left(\frac{1}{p}-1\right)z<x+d$, then \eqref{jump} reads
\begin{align*}
&(x+d)^p-c_p^p(-z)^p+p(x+d-z)(-z)^{p-1}\int_{-(x+d)/z}^\infty \frac{r^{p-1}}{r+1}\mbox{d}r\\
&\qquad \leq \left(\left(\frac{1}{p}-1\right)^p-c_p^p\right)(-z)^p+p(x+d-z)(-z)^{p-1}\int_{p^{-1}-1}^\infty \frac{r^{p-1}}{r+1}\mbox{d}r,
\end{align*}
or
$$ F(s):=\frac{s^p-\left(\frac{1}{p}-1\right)^p}{s+1}-p\int_{p^{-1}-1}^s \frac{r^{p-1}}{r+1}\mbox{d}r\leq 0,$$
with $s=-(x+d)/z\geq p^{-1}-1$. However, the function $F$ vanishes for $s=p^{-1}-1$, and its derivative for $s>p^{-1}-1$ is
$$ F'(s)=-\frac{s^p-(p^{-1}-1)^p}{(s+1)^2}\leq 0,$$
so $F$ is indeed nonpositive and \eqref{jump} holds true. The final case we need to consider is when $-\left(\frac{1}{p}-1\right)z<y<x+d.$ Then \eqref{jump} takes the form
\begin{align*}
&(x+d)^p-c_p^p(-z)^p+p(x+d-z)(-z)^{p-1}\int_{-(x+d)/z}^\infty \frac{r^{p-1}}{r+1}\mbox{d}r\\
&\qquad \leq y^p-c_p^p(-z)^p+p(x+d-z)(-z)^{p-1}\int_{-y/z}^\infty \frac{r^{p-1}}{r+1}\mbox{d}r,
\end{align*}
which can be rewritten as
$$ G(s):=\frac{s^p-\left(-y/z\right)^p}{s+1}-p\int_{-y/z}^s \frac{r^{p-1}}{r+1}\mbox{d}r\leq 0,$$
with $s=-(x+d)/z$. To see that the latter estimate is valid, we observe that $G(-y/z)=0$ and
$$ G'(s)=-\frac{s^p-(-y/z)^p}{(s+1)^2}\leq 0$$
provided $s>-y/z$. This proves the desired bound. 

(iv) If $x\geq z$, then the terms in $U$ involving the appropriate integrals are nonnegative. Therefore, we see that
\begin{equation}\label{majj}
 U(x,y,z)\geq \left(\max\left\{y,-\left(\frac{1}{p}-1\right)z\right\}\right)^p-c_p^p(-z)^p\geq y^p-c_p^p(-z)^p.
\end{equation}
This yields the claim and completes the proof of the lemma.
\end{proof}

\begin{proof}[Proof of \eqref{mainin}] Here the reasoning is similar to that appearing in the proof of \eqref{mainin0}, so we will be brief. 
Pick an arbitrary adapted supermartingale $X=(X_t)_{t\geq 0}$ starting from $0$, which has only nonnegative jumps, and let $\e>0$. Consider the process $Z^\e=((X_t,M_t^+\vee \e, M_t^-\wedge (-\e)))_{t\geq 0}$. By Lemma \ref{lemma1} (i), we are allowed to apply It\^o's formula to $U$ and this process. As the result,  we obtain that for  $t\geq 0$,
\begin{equation}\label{ito5}
U(Z_t^\e)=I_0+I_1+I_2+\frac{I_3}{2}+I_4,
\end{equation}
where
\begin{align*}
I_0&=U(Z_0^\e)=U(0,\e,-\e),\\
I_1&=\int_{0+}^t U_x(Z_{s-}^\e)\mbox{d}X_s,\\
I_2&=\int_{0+}^tU_y(Z_{s-}^\e)\mbox{d}(M^{+c}_{s-}\vee \e)+\int_{0+}^tU_z(Z_{s-}^\e)\mbox{d}(M_{s}^{-}\wedge (-\e)),\\
I_3&=\int_{0+}^t U_{xx}(Z_{s-}^\e)\mbox{d}[X,X]_s^c,\\
I_4&=\sum_{0<s\leq t} \Big[U(Z_s^\e)-U(Z^\e_{s-})-U_x(Z^\e_{s-})\Delta X_s\Big].
\end{align*}
Here $(M^{+c}_{s}\vee \e)_{s\geq 0}$ denotes the continuous part of the process $M^+\vee \e$. Note that because of the appearance of this process in $I_2$, there is no corresponding term $U(Z^\e_{s-})\Delta(M^+_s\vee \e)$ in $I_4$. On the other hand, as in the proof of \eqref{mainin0}, the process $M^-\wedge (-\e)$ is continuous due to the assumption on the sign of the jumps of $X$.

Now, we see that $\E I_1\leq 0$, by the properties of stochastic integrals (see the proof of \eqref{mainin0} for a similar argument). Next, an application of \eqref{uz} gives that the first integral in $I_2$ is zero and the second is nonpositive; again, see the analogous reasoning in the proof of \eqref{mainin0}. The term $I_3$ vanishes, since $U_{xx}$ is zero. Finally, each summand in $I_4$ is nonpositive: this has been just proved in \eqref{jump} above. Therefore, combining all the above facts, we see that
$$ \E U(Z_t^\e)\leq U(0,\e,-\e),$$
or, by virtue of \eqref{maj},
$$ \E \left(M_t^+\vee \e\right)^p\leq c_p^p\E \left((-M_t^-)\vee (-\e)\right)^p+U(0,\e,-\e).$$
It remains to let $\e\to 0$ and  then let $t$ go to infinity. The proof is complete.
\end{proof}

\subsection{Proof of \eqref{mainin2}} Finally, we turn our attention to the bound for the two-sided maximal function. The idea is to proceed exactly in the same manner as in the preceding subsection. What properties should the appropriate special function have? A careful inspection of the above proof shows that it is enough to find $U$ enjoying the conditions of Lemma \ref{lemma1}, with \eqref{maj} replaced by
$$ U(x,y,z)\geq (\max\{y,-z\})^p-\mathfrak{c}_p^p(-z)^p.$$
However, the function $U$ introduced in \S2.2 does have this property when $p\in (0,1/2]$; see the first estimate in \eqref{majj}. Consequently, for these values of $p$, the inequality \eqref{mainin2} follows at once from the reasoning presented previously (note that $c_p=\mathfrak{c}_p$ provided $0<p\leq 1/2$). Thus, it remains to establish the desired bound in the range $(1/2,1)$ only. Suppose that $p$ lies in this interval and consider a function $U:\R \times [0,\infty)\times (-\infty,0]\to \R$ defined by 
$$ U(x,y,z)=(1-\mathfrak{c}_p^p)(-z)^p+p(-z)^{p-1}(x-z)\int_1^\infty \frac{r^{p-1}}{r+1}\mbox{d}r$$
if $y<-z$, and
$$ U(x,y,z)=y^p-\mathfrak{c}_p^p(-z)^p+p(-z)^{p-1}(x-z)\int_{-y/z}^\infty \frac{r^{p-1}}{r+1}\mbox{d}r$$
for remaining $(x,y,z)$. For the sake of completeness, let us list the key properties of this function in a lemma below. We omit the straightforward proof: analogous argumentation has been already presented in the proof of Lemma \ref{lemma1}.

\begin{lemma}\label{lemma2}
The function $U$ enjoys the following properties.

(i) For all $z<0<y$, the function $U(\cdot,y,z)$ is of class $C^2$ and the partial derivatives $U_y(y,y,z)$, $U_z(z,y,z)$ exist.

(ii) For all $z<x<y$ we have
$$ U_x(x,y,z)\geq 0,\quad  U_y(y,y,z)= 0\quad \mbox{and}\quad U_z(z,y,z)\geq 0.$$

(iii) If $z<x<y$, then for any $d\geq 0$ we have
$$ U(x+d,(x+d)\vee y,z)\leq U(x,y,z)+U_x(x,y,z)d.$$

(iv) If $x\geq z$, then we have
$$
U(x,y,z)\geq (y\vee (-z))^p-\mathfrak{c}_p^p(-z)^p.
$$
\end{lemma}

Equipped with this statement, we obtain the proof of \eqref{mainin2} by a word-by-word repetition of the reasoning from \S2.2. Since no additional arguments are needed, we omit the details, leaving them to the reader.

\section{Sharpness}\label{sharpness} 
For the sake of clarity, we have decided to split this section into four parts. Throughout, $B=(B_t)_{t\geq 0}$ denotes a standard, one-dimensional Brownian motion starting from $0$.

\subsection{Optimality of $C_p$ in \eqref{mainin0}, the case $p\in (0, p_0]$} Let $\beta$ be an arbitrary positive number smaller than $p^{-1}-1$ and let $\delta>0$. Let $\tau_0$ be defined by
$$ \tau_0=\inf\{t:B_t\in \{-1,\beta\}\}.$$
Now, define the variable $\sigma$ and the stopping times $\tau_1$, $\tau_2$, $\ldots$ by the following inductive procedure. Suppose that $n$ is a given nonnegative integer. If $B_{\tau_n}=\beta(1+\delta)^n$, then put $\sigma=n$ and $\tau_{n+1}=\tau_{n+2}=\ldots=\tau_n$; on the other hand, if $B_{\tau_n}=-(1+\delta)^n$, then let
$$ \tau_{n+1}=\inf\{t>\tau_n:B_t\in \{-(1+\delta)^{n+1},\beta(1+\delta)^{n+1}\}.$$
To gain some intuition about these random variables, let us look at the behavior of the sequence $(B_{\tau_n})_{n\geq 0}$. The first step is to wait until the Brownian motion reaches $-1$ or $\beta$. If $B_{\tau_0}=\beta$, we stop the evolution.  If $B_{\tau_0}=-1$, then we start the second stage and continue until $B_t$ reaches $-(1+\delta)$ or $\beta(1+\delta)$.  If the second case occurs we stop but if $B_{\tau_1}=-(1+\delta)$, then we start the third stage and wait until $B$ reaches $-(1+\delta)^2$ or $\beta(1+\delta)^2$.  This  pattern is then repeated. We define the random variable  $\sigma$ to be the number of nontrivial stages which occur \emph{before} the Brownian motion stops. Using  elementary properties of  Brownian motion, we see that
\begin{equation}\label{s00}
 \mathbb{P}(\sigma=0)=\frac{1}{\beta+1}
\end{equation}
and, for any nonnegative integer $n$,
\begin{equation}\label{s0m}
 \mathbb{P}(\sigma>n)=\frac{\beta}{\beta+1}\left(\frac{\beta(1+\delta)+1}{(\beta+1)(1+\delta)}\right)^n.
\end{equation}
Hence, in particular, for any $n=1,\,2,\,\ldots$ we have
\begin{equation}\label{s0n}
\mathbb{P}(\sigma=n)=\frac{\beta}{\beta+1}\left(\frac{\beta(1+\delta)+1}{(\beta+1)(1+\delta)}\right)^{n-1}\frac{\delta}{(\beta+1)(1+\delta)}.
\end{equation}
Consequently, we see that $\eta$, the pointwise limit of the sequence $(\tau_n)_{n\geq 0}$, is finite almost surely. Put $X=(B_{\eta\wedge t})_{t\geq 0}$ and let us derive the $p$-th norms of $X_\eta$ and $M^-$. By the very construction, $X_\eta = \beta(1+\delta)^n$ on the set $\{\sigma=n\}$, so by \eqref{s00} and \eqref{s0n},
\begin{align*} 
||X_\eta||_p^p&=\E |X_\eta|^p\\
&=\frac{\beta^p}{\beta+1}+\sum_{n=1}^\infty (1+\delta)^{np}\cdot \frac{\beta}{\beta+1}\left(\frac{\beta(1+\delta)+1}{(\beta+1)(1+\delta)}\right)^{n-1}\frac{\delta}{(\beta+1)(1+\delta)}\\
&=\frac{\beta^p}{\beta+1}+\frac{\beta(1+\delta)^{p-1}\delta}{(\beta+1)^2}\sum_{n=1}^\infty \left(\frac{(\beta(1+\delta)+1)(1+\delta)^{p-1}}{\beta+1}\right)^{n-1}.
\end{align*}
Now observe that
\begin{align*}
 \frac{(\beta(1+\delta)+1)(1+\delta)^{p-1}}{\beta+1}&=(1+\delta)^{p-1}+\frac{\beta\delta(1+\delta)^{p-1}}{\beta+1}\\
&=1+\left(p-1+\frac{\beta}{\beta+1}\right)\delta+O(\delta^2)
\end{align*}
as $\delta\to 0$. Since $\beta<p^{-1}-1$, the above expression is less than $1$ for small $\delta$.  Thus $||X_\eta||_p$ is finite. On the other hand, it follows directly from the construction that $M^-\geq -(1+\delta)^n$ on $\{\sigma=n\}$, that is, we have the pointwise bound $-\beta M^-\leq X_\eta$. This gives $||X||_p\geq ||X_\eta||_p\geq \beta ||M^-||_p$, and since $\beta<p^{-1}-1$ was arbitrary, the optimal constant in \eqref{mainin0} cannot be smaller than $p^{-1}-1$.

\subsection{Optimality of $C_p$ in \eqref{mainin0}, the case $p\in (p_0,1)$} Let $\beta$ be an arbitrary parameter smaller than $\alpha_p^{1/(1-p)}$. Fix a large positive constant $K$, an even larger integer $N$ and set $\delta=K/N$. Let $\tau_0$, $\tau_1$, $\tau_2$, $\ldots$ and $\sigma$ be as in preceding case. The main difference in comparison to the previous construction is that we put $X=(B_{\tau_N\wedge t})_{t\geq 0}$.  That is, we terminate the Brownian motion after at most $N$ stages. If $n\leq N$, then $-M^-\geq (1+\delta)^n$ and $X_\eta=\beta(1+\delta)^n$ on $\{\sigma=n\}$.  Furthermore, 
$ -M^-=|X_\eta|=(1+\delta)^{N+1}$ on $\{\sigma>N\}$. Consequently, using \eqref{s00}, \eqref{s0m} and \eqref{s0n}, we derive that
\begin{align*}
 ||M^-||_p^p&\leq \frac{\beta \delta(1+\delta)^{p-1}}{(\beta+1)^2} \sum_{n=0}^N\left(\frac{(\beta(1+\delta)+1)(1+\delta)^{p-1}}{\beta+1}\right)^{n-1}\\
 &\quad +\frac{\beta(1+\delta)^{(N+1)p}}{\beta+1}\left(\frac{\beta(1+\delta)+1}{(\beta+1)(\delta+1)}\right)^N\\
 &=\frac{\beta\delta(1+\delta)^{p-1}}{\beta+1}\cdot \frac{1-\left(\frac{(\beta(1+\delta)+1)(1+\delta)^{p-1}}{\beta+1}\right)^N}{(\beta+1)(1-(1+\delta)^{p-1}) -\beta\delta(1+\delta)^{p-1}}\\
 &\quad +\frac{\beta(1+\delta)^p}{\beta+1}\left(\frac{(\beta(1+\delta)+1)(1+\delta)^{p-1}}{\beta+1}\right)^N.
\end{align*}
A similar computation shows that
\begin{align*}
 ||X_\eta||_p^p&=\beta^p\cdot \frac{\beta\delta(1+\delta)^{p-1}}{\beta+1}\cdot \frac{1-\left(\frac{(\beta(1+\delta)+1)(1+\delta)^{p-1}}{\beta+1}\right)^N}{(\beta+1)(1-(1+\delta)^{p-1}) -\beta\delta(1+\delta)^{p-1}}\\
 &\quad +\frac{\beta(1+\delta)^p}{\beta+1}\left(\frac{(\beta(1+\delta)+1)(1+\delta)^{p-1}}{\beta+1}\right)^N.
\end{align*}
Now let $N$ go to infinity (then $\delta=K/N$ converges to $0$). The upper bound for $||M^-||_p^p$ converges to
$$ \frac{\beta}{\beta+1}\frac{1-e^{pK-K/(\beta+1)}}{1-p-\beta_p}+\frac{\beta}{\beta+1}e^{pK-K/(\beta+1)},$$
while $||X_\eta||_p^p$ tends to
$$ \beta^p\cdot \frac{\beta}{\beta+1}\frac{1-e^{pK-K/(\beta+1)}}{1-p-\beta_p}+\frac{\beta}{\beta+1}e^{pK-K/(\beta+1)}.$$
Consequently, the optimal constant in \eqref{mainin0} cannot be smaller than
$$ \frac{\beta^p(1-e^{pK-K/(\beta+1)})/(1-p-\beta_p)+e^{pK-K/(\beta+1)}}{(1-e^{pK-K/(\beta+1)})/(1-p-\beta_p)+e^{pK-K/(\beta+1)}},$$
for any $K$. Letting $K\to \infty$, we easily see that the expression above converges to
$$ \frac{\beta^p+\beta p+p-1}{(\beta+1)p}=1+\frac{\beta^p-1}{(\beta+1)p}.$$
However, recall that $\beta$ was an arbitrary positive constant smaller than $\alpha_p^{1/(1-p)}$. Letting $\beta\to \alpha_p^{1/(1-p)}$ and using the definitions of $\alpha_p$ and $C_p$, we see that the expression above converges to $C_p^p$. This proves the sharpness of the estimate \eqref{mainin0}.

\subsection{Optimality of $c_p$ in \eqref{mainin}} Here the optimal stopping procedure will be more complicated. 
Let $\beta$ be a given positive number smaller than $p^{-1}-1$ and let $\delta>0$. Define the stopping times $\tau_0$, $\tau_1$, $\tau_2$, $\ldots$ and the variable $\sigma$ with the use of the same formulas as in the preceding cases. We will also need an additional stopping time $\eta$  given as follows: if $\sigma=n$ (and hence $B_{\tau_n}=\beta(1+\delta)^n$), then put 
$$ \eta=\inf\{t:B_t=-(1+\delta)^n\}.$$
We easily check that $\eta$ is a stopping time which is finite almost surely. Put $X=(B_{\eta\wedge t})_{t\geq 0}$ and let us compute the norms $||M^+||_p$ and $||M^-||_p$. By the above construction, we see that $M^-_{\eta}=-(1+\delta)^n$ on $\{\sigma=n\}$, so $\E (-M_{\eta}^-)^p1_{\{\sigma=n\}}=(1+\delta)^{np}\mathbb{P}(\sigma=n)$. Therefore, by \eqref{s00} and \eqref{s0n}, we have 
\begin{align*}
 \E (-M_{\eta}^-)^p&=\sum_{n=0}^\infty (1+\delta)^{np}\mathbb{P}(\sigma=n)\\
 &=\frac{\beta}{\beta+1}+\sum_{n=1}^\infty (1+\delta)^{np}\frac{\beta}{\beta+1}\left(\frac{\beta(1+\delta)+1}{(\beta+1)(\delta+1)}\right)^{n-1}\frac{\delta}{(\beta+1)(\delta+1)}\\
 &=\frac{\beta}{\beta+1}\left\{1+\frac{\delta}{\beta(1+\delta)+1}\sum_{n=1}^\infty \left[\frac{(1+\delta)^{p-1}(\beta(1+\delta)+1)}{\beta+1}\right]^n \right\}\\
 &=\frac{\beta}{\beta+1}\left\{1+ \frac{\delta(1+\delta)^{p-1}}{(\beta+1)\left(1-(1+\delta)^{p-1}(1+\frac{\beta\delta}{\beta+1})\right)}\right\}.
\end{align*}
Now, if we let $\delta$ go to $0$, we see that
\begin{equation}\label{seeabove}
 \E (-M_{\eta}^-)^p\to \frac{\beta}{\beta+1}\left\{1+\frac{1}{(\beta+1)(1-p-\frac{\beta}{\beta+1})}\right\}=\frac{\beta(2-p-\beta p)}{(\beta+1)(1-p-\beta p)}.
\end{equation}
The analysis of $\E (M_{\eta}^+)^p$ is slightly more complicated. Suppose that $\sigma=n$.  Then $B_{\tau_n}=\beta(1+\delta)^n$ and, using elementary properties of Brownian motion, we see that for each $y>\beta(1+\delta)^n$,
\begin{align*}
 \mathbb{P}(M_\eta^+\geq y|\sigma=n)&=\mathbb{P}(B\mbox{ reaches }y\mbox{ before it reaches }-(1+\delta)^n)\;|\;\sigma=n)\\
&=\frac{(\beta+1)(1+\delta)^n}{y+(1+\delta)^n}
\end{align*}
and hence the density of $M_\eta^+$, given that $\sigma=n$, is equal to 
$$ g(s)=\frac{(\beta+1)(1+\delta)^n}{(s+(1+\delta)^n)^2},\qquad s>\beta(1+\delta)^n.$$
Consequently,
$$ \E \left[(M_\eta^+)^p|\sigma=n\right]=\int_{\beta(1+\delta)^n}^\infty \frac{s^p(\beta+1)(1+\delta)^n}{(s+(1+\delta)^n)^2}\mbox{d}s=(\beta+1)(1+\delta)^{pn}
\int_\beta^\infty \frac{s^p}{(s+1)^2}\mbox{d}s$$
and hence by \eqref{s0n} we obtain, after some straightforward manipulations,
\begin{align*}
\E (M_\eta^+)^p&\geq \frac{\beta(1+\delta)^{p-1}\delta}{\beta+1}\int_\beta^\infty \frac{s^p}{(s+1)^2}\mbox{d}s \cdot \sum_{n=1}^\infty \left(\frac{(\beta(1+\delta)+1)(1+\delta)^{p-1}}{\beta+1}\right)^{n-1}\\
&= \frac{\beta(1+\delta)^{p-1}\delta}{\beta+1}\int_\beta^\infty \frac{s^p}{(s+1)^2}\mbox{d}s \cdot \frac{1}{1-(1+\delta)^{p-1}(1+\frac{\beta}{\beta+1}\delta)}.
\end{align*}
When $\delta$ goes to $0$, the latter expression converges to
$$ \int_\beta^\infty \frac{s^p}{(s+1)^2}\mbox{d}s \cdot \frac{\beta}{1-p-\beta p}.$$
Putting all the above facts together, we see that
$$ \liminf_{\delta \to 0} \frac{\E (M_\eta^+)^p}{\E (-M_{\eta}^-)^p}\geq \frac{(\beta+1) \int_\beta^\infty \frac{s^p}{(s+1)^2}\mbox{d}s}{2-p-\beta p}.$$
However, $\beta$ was an arbitrary positive number smaller than $1/p-1$. If we let $\beta$ go to $1/p-1$ in the above expression on the right, we see that the optimal constant in \eqref{mainin} cannot be smaller than
$$ \left(\frac{1}{p}\int_{p^{-1}-1}^\infty \frac{s^p}{(s+1)^2}\mbox{d}s\right)^{1/p}.$$
This is precisely $c_p$, which can be easily verified with the use of integration by parts.

\subsection{Optimality of $\mathfrak{c}_p$ in \eqref{mainin2}} If $0<p\leq 1/2$, then the sharpness of \eqref{mainin2} follows at once from \S3.3, since \eqref{mainin2} is stronger than \eqref{mainin}. Therefore it is enough to study the case $1/2<p<1$ only. The calculations are very similar to those in the preceding section; however, some small but nontrivial changes are required, so we have decided to present the details. Let $\tau_0$, $\tau_1$, $\tau_2$, $\ldots$, $\sigma$, $\eta$ be defined by the same formulas (for some fixed $\beta<p^{-1}-1$) and put $X=(B_{\eta\wedge t})_{t\geq 0}$. Now, for a given integer $n$ and $0<y\leq (1+\delta)^n$, we see that
$$ \mathbb{P}\big(M_\eta\geq y|\sigma=n\big)=1,$$
since $M_{\eta}^-=-(1+\delta)^n$ on the set $\{\sigma=n\}$. For $y>(1+\delta)^n$ we have, as previously,
$$ \mathbb{P}\big(M_\eta\geq y|\sigma=n\big)=\frac{(\beta+1)(1+\delta)^n}{y+(1+\delta)^n}.$$
Consequently, we derive that the conditional $p$-th moment of $M_\eta$ is equal to
\begin{align*}
 &\E [M_\eta^p |\sigma=n]\\
&\qquad \qquad =(1+\delta)^{np}\cdot \frac{(1+\delta)^n(1-\beta)}{2(1+\delta)^n}+\int_{(1+\delta)^n}^\infty \frac{s^p(\beta+1)(1+\delta)^n}{(s+(1+\delta)^n)^2}\mbox{d}s\\
&\qquad \qquad =\frac{1-\beta}{2}(1+\delta)^{np}+(\beta+1)(1+\delta)^{np}\int_1^\infty \frac{s^p}{(s+1)^2}\mbox{d}s.
\end{align*}
Therefore, by \eqref{s0n},
\begin{align*}
& \E M_\eta^p\\
&\geq \sum_{n=1}^\infty \left[\frac{1-\beta}{2}+(\beta+1)\int_1^\infty \frac{s^p}{(s+1)^2}\mbox{d}s\right]
\frac{\beta \delta(1+\delta)^{np-1}}{(\beta+1)^2}\left(\frac{\beta(1+\delta)+1}{(\beta+1)(1+\delta)}\right)^{n-1}\\
&=\left[\frac{1-\beta}{2}+(\beta+1)\int_1^\infty \frac{s^p}{(s+1)^2}\mbox{d}s\right]\frac{\beta(1+\delta)^{1-p}}{(\beta+1)\big((\beta+1)(1+\delta)^{1-p}-\beta(1+\delta)-1\big)},
\end{align*}
where the geometric series converges due to the assumption $\beta<p^{-1}-1$. 
Letting $\delta\to 0$, we see that the latter expression converges to
$$ \frac{\beta}{(\beta+1)(1-p-\beta p)}\left[\frac{1-\beta}{2}+(\beta+1)\int_1^\infty \frac{s^p}{(s+1)^2}\mbox{d}s\right].$$
Thus, we infer from \eqref{seeabove} that the upper bound for the ratio 
$ \E M_\eta^p/\E (-M^-_{\eta})^p$ cannot be smaller than
$$ \frac{\frac{\beta}{(\beta+1)(1-p-\beta p)}\left[\frac{1-\beta}{2}+(\beta+1)\int_1^\infty \frac{s^p}{(s+1)^2}\mbox{d}s\right]}{\frac{\beta(2-p-\beta p)}{(\beta+1)(1-p-\beta p)}}=\frac{\frac{1-\beta}{2}+(\beta+1)\int_1^\infty \frac{s^p}{(s+1)^2}\mbox{d}s}{2-p-\beta p}.$$
It suffices to note that the latter expression converges to $\mathfrak{c}_p^p$ as $\beta\to p^{-1}-1$. This establishes the desired sharpness.

\section{Harmonic functions in domains of $\R^n$} \label{applications} 

In \cite{Bur1} Burkholder proves an interesting version of inequalities \eqref{burk1} and \eqref{burk2} for harmonic functions in the upper half-space $\R^{n+1}_{+}=\{(x, y): x\in \R^n, y>0\}$.  We briefly recall his result.  If $u$ is harmonic in $\R^{n+1}_{+}$, we let $N_{\alpha}(u)$ be its non-tangential maximal function defined by 
$$
N_{\alpha}(u)(x)=\sup\{|u(s, y)|: (s, y)\in \Gamma_{\alpha}(x)\},
$$
where $\Gamma_{\alpha}(x)=\{(s, y): |x-s|<\alpha y \}$ is the cone in $\R^{n+1}_{+}$ of aperture $\alpha$.  Setting $u^{-}=u\wedge 0$, we define the corresponding one-sided variant of the above object by
$$
N^{-}_{\alpha}(u)(x)=\sup\{-u^-(s, y): (s, y)\in \Gamma_{\alpha}(x)\}. 
$$
\begin{theorem}[Burkholder \cite{Bur1}]\label{Burkholder2}  Suppose $u$ is harmonic in $\R^{n+1}_{+}$ satisfying $u(0, y)=o(y^{-n})$,  as $y\to\infty$.  If $\Phi$ is as in Theorem \ref{Burkholder} then 
\begin{equation}\label{burk3}
\int_{\R^n}\Phi\big(N_{\alpha}(u)(x)\big) dx\leq C\int_{R^n} \Phi\big(-N^{-}_{\alpha}(u)(x)\big) dx,
\end{equation}
for some constant $C$ depending on $\Phi$, $n$ and $\alpha$.  In particular, 
\begin{equation}\label{burk4} 
\|N_{\alpha}(u)\|_p\leq C_{p, \alpha, n}\|N^{-}_{\alpha}(u)\|_p, \,\,\,\,\,\, 0<p<1. 
\end{equation}
\end{theorem}

It is shown in \cite[p.451]{Bur1} that this inequality fails for $p\geq 1$.  

A similar result holds for harmonic functions in the ball of $\R^n$ with the normalization $u(0)=0$.  Using Theorem \ref{theorem3} and the classical fact that the composition of a superharmonic function with a Brownian motion is a supermartingale (see Doob \cite{Do}), we obtain the following probabilistic version of Burkholder's result.

\begin{theorem}\label{theomrem3}  Let $D\subset \R^n$ be a domain (an open connected set).  Fix a point $x_0\in D$ and let $B=(B_t)_{t\geq 0}$ be Brownian motion starting at $x_0$ and killed upon leaving $D$.  Denote by $\tau_D$ its exit time from $D$. Assume further that $u$ is a superharmonic function in $D$ satisfying the normalization condition $u(x_0)=0$.  Define $M(u)=\sup_{t\geq 0}|u(B_{t\wedge \tau_D})|$ and  $M^{-}(u)=\inf_{t\geq 0}u^{-}(B_{t\wedge \tau_D})$.  Then 
\begin{equation}\label{harmonic-prob} 
\|M(u)\|_p\leq \mathfrak{c}_p\|M^{-}(u)\|_p, \,\,\,\,\,\, 0<p<1,  
\end{equation}
where $\mathfrak{c}_p$ is the constant in Theorem \ref{theorem3}. 
\end{theorem}

This inequality has an interesting application for harmonic functions in the unit disc  $\D=\{z\in \bC: |z|<1\}$ in the plane.   Suppose $u$ is harmonic in $\D$ and, as in the upper half-space, define $N_{\alpha}(u)(e^{i\theta})$ and 
$N^{-}_{\alpha}(u)(e^{i\theta})$ where this time the supremum is taken over the Stoltz domain given by the interior of the smallest convex set containing the disc $\{z\in \bC: |z|<\alpha\}$ and the point $e^{i\theta}$. (Here, we assume $0<\alpha<1$.) It is proved in Burkholder, Gundy and Silverstein \cite{BurGunSil} that there exists a  constant $k_{\alpha}$  depending only on $\alpha$ such that
\begin{equation}\label{dist1}
 m\{\theta: N_{\alpha}(u)(e^{i\theta})>\lambda\}\leq k_{\alpha}\mathbb{P}\big(M(u)>\lambda\big), 
\end{equation}
for all $\lambda>0$.  Here $m$ denotes the Lebesgue measure on the circle.  While the opposite inequality is stated in \cite{BurGunSil} for harmonic functions, it actually holds for subharmonic functions (see also Durrett \cite[p.137]{Dur}) and we have that there exists a constant $K_{\alpha}$ (again depending only on $\alpha$) such that
\begin{equation}\label{dist2}
 \mathbb{P}\big(-M^{-}(u)>\lambda\big)\leq K_{\alpha}m\{\theta: N^{-}_{\alpha}u(e^{i\theta})>\lambda\}.
\end{equation} 
Combining \eqref{harmonic-prob}, \eqref{dist1} and \eqref{dist2} we obtain 

\begin{corollary} Let $u$ be a harmonic function in the unit disk $D$ with $u(0)=0$. Then 
\begin{equation}\label{disc1} 
\|N_{\alpha}(u)\|_p\leq k_{\alpha}K_{\alpha}\mathfrak{c}_p\|N^{-}_{\alpha}(u)\|_p, \,\,\,\,\,\, 0<p<1,  
\end{equation} 
where the constants $k_{\alpha}$, $K_{\alpha}$ and $\mathfrak{c}_p$ are, respectively,   those appearing in \eqref{dist1}, \eqref{dist2} and Theorem  \ref{theorem3}.  In particular, the dependence on $p$ in the harmonic function inequality \eqref{disc1}  is the same as in the martingale inequality \eqref{mainin2}.  A similar inequality holds for harmonic functions on the upper half-space $\R^{2}_{+}$ satisfying the hypothesis of Theorem \ref{Burkholder2}. 
\end{corollary}


\begin{thebibliography}{99}

\bibitem{Bur1} D.L. Burkholder, {\it One-sided maximal functions and $H^p$,} J. Funct. Anal. {\bf 18} (1975), 429-454. 

\bibitem{BurGunSil}  D.L. Burkholder, R. F. Gundy and M. L. Silverstein, {\it A maximal function characterization of the class $H^p$,}  Trans. Amer. Math. Soc. {\bf 157} (1971), 137-153.

\bibitem{Do}{J. L. Doob, Stochastic processes, Wiley, New York, 1953.}

\bibitem{Dur} R. Durrett, {\em Brownian Motion and Martingales
in Analysis},
Wadsworth,  Belmont, CA, 1984.


\bibitem{Gi}{D. Gilat, {\it The best bound in the LlogL inequality of Hardy and Littlewood and its martingale counterpart}, Proc. Amer. Math. Soc. \textbf{97} (1986), 429--436.}
\bibitem{Pe}{G. Peskir, {\it Optimal stopping of the maximum process: The maximality principle}, Ann. Probab. \textbf{26} (1998), 1614--1640.}
\bibitem{PS}{G. Peskir and A. Shiryaev, {\it Optimal stopping and free boundary problems}, Lect. Notes in Math. ETH Zurich}

\bibitem{Pro} P. E. Protter.
{\em Stochastic integration and differential equations},
Springer-Verlag, Berlin, second edition, 2004.

\bibitem{Sch} M. Scheutzow, 
{\it A stochastic Gronwall lemma,} Infin. Dimens. Anal. Quantum Probab. Relat. Top., {\bf 16} (2013), 4 pages. 
\bibitem{Shao}{Q.-M. Shao, {\it A comparison theorem on maximal inequalities between negatively associated and
independent random variables}, J. Theoret. Probab. \textbf{13} (2000), 343--356.}
\end{thebibliography}
\end{document}